%% file: SHL_torus_ver_9.tex
\documentclass[12pt, one sided]{amsart}
\RequirePackage[colorlinks,citecolor=blue,urlcolor=blue]{hyperref}

\usepackage{latexsym,amssymb,amsmath,amsfonts,amsthm}
\usepackage{amsthm,amsmath}
\usepackage{float}
\usepackage{enumerate}
\usepackage[margin=3cm]{geometry}
\usepackage{bbm}
\usepackage{subfigure}
\usepackage{graphicx}
\usepackage[toc,page]{appendix}
\usepackage{multirow}
\usepackage{amsfonts}
\usepackage[all]{xy}
\usepackage{caption}
\usepackage{geometry}                
\usepackage{graphicx}
\usepackage{amssymb}
\usepackage{epstopdf}
\usepackage{color}
\usepackage{bbm, dsfont}
\usepackage{hyperref}
\usepackage{cleveref}
\usepackage[utf8]{inputenc}
\usepackage{amssymb,amsthm,amsmath,amssymb,wrapfig,dsfont}
\usepackage[dvipsnames]{xcolor}
\definecolor{myred}{RGB}{251,154,133}
\definecolor{myblue}{RGB}{153,206,227}
\definecolor{mylightblue}{RGB}{0, 150, 255}
\definecolor{mygreen}{RGB}{32, 210, 64}
\definecolor{mygray}{RGB}{220, 220, 220}

\usepackage{tikz}
\usetikzlibrary{decorations.pathmorphing}
\tikzset{snake it/.style={decorate, decoration=snake}}
\usetikzlibrary{shapes.geometric,positioning,decorations.pathreplacing} 

\usepackage{breqn}
\usepackage{upgreek}
\usepackage{enumitem}  
\usepackage{graphicx} 
\graphicspath{ {/Users/annazhuchenko/Dropbox/Anna - SHL in a torus/SHL_torus_paper_ver_1/Graphics} }
\usepackage{multibib}

\newtheorem{theorem}{Theorem}
\newtheorem{definition}{Definition}[section]
\newtheorem{lemma}[definition]{Lemma}
\newtheorem{remark}[definition]{Remark}
\newtheorem{proposition}[definition]{Proposition}
\newtheorem{corollary}[definition]{Corollary}

\newtheorem{question}{Question}

\def\beq{ \begin{equation} }
	\def\eeq{ \end{equation} }

\def\ep{\varepsilon}

\def\square{\vcenter{\vbox{\hrule height .4pt
			\hbox{\vrule width .4pt height 5pt \kern 5pt
				\vrule width .4pt} \hrule height .4pt}}}

\newcommand{\BC}{{\mathbb{C}}}
\newcommand{\BD}{{\mathbb{D}}}

\newcommand{\BH}{{\mathbb{H}}}

\newcommand{\BN}{{\mathbb{N}}}

\newcommand{\BR}{{\mathbb{R}}}

\newcommand{\BT}{{\mathbb{T}}}

\newcommand{\BZ}{{\mathbb{Z}}}

\newcommand{\CF}{{\mathcal{F}}}

\newcommand{\ind}{{\mathbbm{1}}}

\newcommand{\prob}{{\bf P}}

\newcommand{\bae}{\begin{equation}\begin{aligned}}
		\newcommand{\eae}{\end{aligned}\end{equation}}
\newcommand{\ev}{\mathbf{E}}

\DeclareFontFamily{OML}{rsfs}{\skewchar\font'177}
\DeclareFontShape{OML}{rsfs}{m}{n}{ <5> <6> rsfs5 <7> <8> <9>
	rsfs7 <10> <10.95> <12> <14.4> <17.28> <20.74> <24.88> rsfs10 }{}
\DeclareMathAlphabet{\mathfs}{OML}{rsfs}{m}{n}

\newcommand{\HL}{\text{HL}}

\newcommand{\ima}{\text{Im}}

\newcommand{\slitN}{\mathfs{S}^{N,\delta}}
\newcommand{\slitM}{\mathfs{S}^{M,\delta}}
\newcommand{\slit}{\varphi}
\newcommand{\SM}{\varphi}

\newcommand{\scal}{\delta (N,\lambda)}

\newcounter{relctr} 
\everydisplay\expandafter{\the\everydisplay\setcounter{relctr}{0}} 

\newcommand*{\hm}[1]{#1\nobreak\discretionary{}%
	{\hbox{$\mathsurround=0pt #1$}}{}}

\begin{document}
	
\title[Cylindrical Hastings Levitov]{Cylindrical Hastings Levitov}

\begin{abstract}
We define a Hastings-Levitov$(0)$ process on a cylinder and prove that the process converges to Stationary Hastings Levitov$(0)$ under appropriate particle size scaling that depends on the radius of the cylinder. The Stationary Hastings Levitov$(0)$ was shown by Berger, Procaccia and Turner to admit tight particle sizes, without a priori particle size normalization, thus it serves as a good model for the phenomenon of diffusion limited aggregation. Technical challenge, in this paper, is in taking the spatial limit together with the correct slit map normalization. This result also shows that the early life of the Hastings Levitov$(0)$ process in the small particle limit, spatially scaled so the slits have unit length, behaves like the Stationary Hastings Levitov$(0)$.
\end{abstract}

\author{Eviatar B. Procaccia}
\address[Eviatar B. Procaccia]{Technion - Israel Institute of Technology}
\urladdr{http://procaccia.net.technion.ac.il}
\email{procaccia@technion.ac.il}
	
\author{Anna Zhuchenko}
\address[Anna Zhuchenko]{Technion - Israel Institute of Technology}
\email{annazhu@campus.technion.ac.il}

\maketitle

\section{Introduction}

	Recently, Berger, Procaccia, and Turner \cite{berger2020growth,procaccia2021dimension} constructed a stationary version of the Hastings-Levitov model (SHL$(0)$) defined on the upper half-plane. They showed that unlike the Hastings Levitov process grown on a disk \cite{norris2012hastings}, particle sizes are tight without normalization, and thus SHL$(0)$ is a good candidate for an off-lattice version of \textit{stationary} DLA \cite{procaccia2018stationary,procaccia2020stationary,procaccia2019stationary,procaccia2021sets}. Moreover, the study provides results for the growth rate of the process, e.g., trees with $n$ particles reach a height of order $n^{2/3}$. This corresponds to the growth upper bound of Kesten \cite{MR915132,Kesten1987HowLA} and predictions of Meakin for DLA growth on a fiber \cite{meakin1983diffusion}. Given these exact results, a natural question is whether one can obtain such control in a finite system?

	In this paper, we construct rigorously SHL on a cylinder of finite width, denoted  CHL$^N$ (Figure \ref{Fig:simulation2}), and prove that as the width of the cylinder tends to infinity, the process converges to SHL$(0)$. This is an off-lattice version of the result in \cite{mu2019scaling}, where it is shown that the classical discrete DLA initiated in a long line segment or a discrete cylinder converges to the Stationary DLA process as the length of the initial line tends to infinity. This allows to use the exact growth bounds proved in \cite{berger2020growth} for the more physical spatially finite setting of SHL on a cylinder. Moreover, the main result of this paper also shows that HL$(0)$ in the small particle limit $\delta\to0$, grown up to time $t\delta^{-1}$, scaled around a point in $\BH$ such that particles are of size 1 (i.e. the curvature of the disk vanishes under this scaling), converges to SHL$(0)$.
	\begin{figure}[hbpt]
		\includegraphics[scale=0.6]{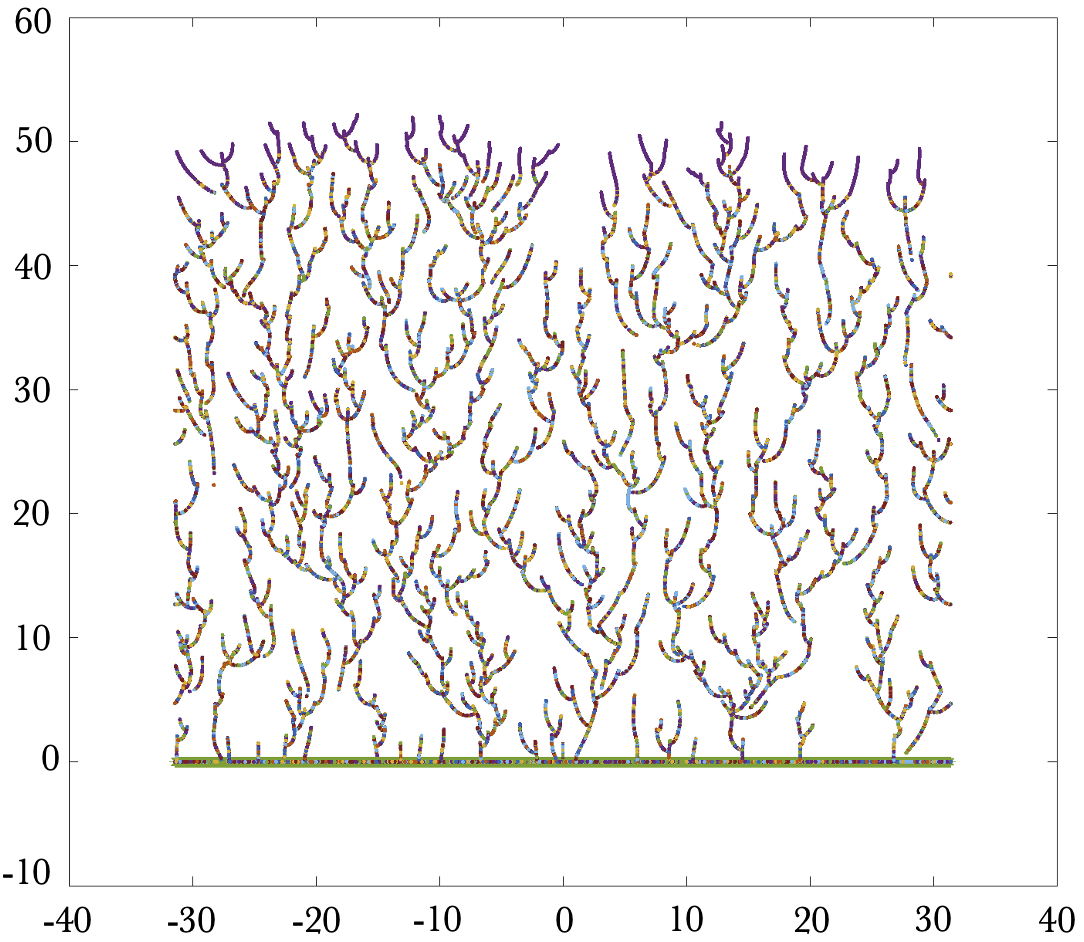}
		\centering
		\caption{Simulation of CHL$^N(0)$ in cylinder of the radius $10$}
		\label{Fig:simulation2}
	\end{figure}

\subsection{Open problems} We finish this section with some questions we find interesting and worthy of further study.

\begin{question}
One can define CHL$^N(\alpha)$, by normalizing the slit sizes according to the conformal radius. In \cite[Question 1]{berger2020growth} it is conjectured that SHL$(\alpha)$ is not well defined for $\alpha\in[1,2]$. Though not a full proof, it would be indicative to show that CHL$^N(\alpha)$ has no natural limiting process as $N\to\infty$, for any $\alpha\in[1,2]$.
\end{question}

\begin{question}
DLA in a cylinder enjoys vertical stationarity, since at each time with positive probability a horizontal line is formed within finitely many steps, regenerating the process. Does vertical stationarity hold for the CHL$^N$? 
\end{question}

\section{Stationary Hastings-Levitov}
For self containment and uniformity of definitions we recall the definition of the SHL$(0)$ appearing in \cite{berger2020growth}. 

Let $\mathcal P$ be the space of discrete measures on $[0,\infty)\times\BR$, equipped with the Borel $\sigma$-algebra corresponding to the weak topology w.r.t. continuous, compactly supported functions on $[0,\infty)\times\BR$. Let $\mathcal Z$ be the space of conformal maps from $\BH$ onto subsets of $\BH$, equipped with the topology of uniform convergence on compact subsets. Let $\mathcal X$ be the space of c\'adl\'ag functions from $[0,\infty)$ to $\mathcal Z$, with the Borel $\sigma$-algebra induced by the Skorohod topology (take some metric which admits the topology in $\mathcal Z$). For $F$ in $\mathcal X$ we write $F_t$ for the function at time $t$, and $F_t(z)$ for the value that this function takes at $z$. Occasionally we will talk about $t\mapsto F_t(z)$; this is the trajectory of the point $z$ as time progresses.

We let $\Omega={\mathcal P}\times{\mathcal X}$, with the product $\sigma$-algebra $\mathcal F$.
We denote an element $\omega$ of $\Omega$ as a pair $(P,F)$ where $P\in {\mathcal P}$ and $F:[0,\infty)\to \mathcal Z$.
A measure $\mu$ on $(\Omega, \mathcal F)$ is called a SHL$(0)$ process if it satisfies a number of requirements. Before stating those requirements, we need to define the particle map by which particles are added to the cluster.

\begin{definition}\label{def:slit_map}
The {\em slit map} $\varphi:\BH \to \BH\setminus [0,i\lambda]$ is defined as $\SM^\lambda(z)=\sqrt{z^2 - \lambda^2}$. For $x \in \BR$, the {\em slit map at $x$}, denoted $\varphi_x$, is the conjugation of $\SM$ by the shift in $x$, namely 
$\SM^\lambda_x(z)=x+\SM^\lambda(z-x)$. 
\end{definition}

\begin{definition}\label{def:BSHL}
A measure $\tilde{\mu}$ on $(\Omega, \mathcal F)$, corresponding to pair $(\tilde{P}, \tilde{F})$, is called a backward SHL$(0)$ process if it satisfies the following requirements. 
\begin{enumerate}
\item\label{item:back_marg_poisson}
(Poisson arrivals) The $\tilde{\mu}$ marginal distribution of $\tilde{P}$ is that of an intensity $1$ Poisson process.

\item\label{item:back_start_id}
(Initial condition) $\tilde{\mu}$-almost surely $\tilde{F}_0$ is the identity.

\item\label{item:back_adapt} 
(Adapted) For every $0 \leq s<t$, $ \tilde{F}_t \circ \tilde{F}_{s}^{-1} $ is $\tilde \CF_{s,t}$-measurable, where $\tilde \CF_{s,t} = \sigma(\tilde P|_{(s,t]\times\BR})$.

\item\label{item:back_not_move_between_jumps}
(Growth condition)
 Let $\tilde{A}$ be the set of Poisson points, i.e. the atoms of $\tilde{P}$, and let $\tilde{A}_{t}=\tilde{A}\cap\{(s,x) : s\leq t\}$ and
 $\tilde{A}_{t,n}=\tilde{A}\cap\{(s,x):|x|\leq n\ ;\ s\leq t\}$. Then $\tilde{\mu}$-almost surely, for every $t\in[0,\infty)$ and $z\in\BH$,
\begin{equation}\label{eq:back_around_middle}
\tilde{F}_t(z)= z + \lim_{n\to\infty} \sum_{(s,x)\in \tilde{A}_{t,n}} \left[ \slit_x^\lambda(\tilde{F}_{s-}(z)) - \tilde{F}_{s-}(z) \right].
\end{equation}

\label{item:back_last}
\end{enumerate}
\end{definition}
Note that we give the definition of the backward SHL$(0)$ as it is more amenable for analysis. However, for any fixed time the processes are equidistributed, and thus it is enough to consider the backward processes. More precisely, define $(P,F)$ to be SHL$(0)$ in a similar way to the backward SHL$(0)$ only changing the Adapted condition to $\forall0\le s<t$, $F_s^{-1}\circ F_t$ is $\sigma(P|_{(s,t]\times\BR})$ measurable, and the growth condition becomes 
$$
F_t(z)=z+ \lim_{n\to\infty} \sum_{(s,x)\in \tilde{A}_{t,n}} \left[ {F}_{s-}(\slit_x^\lambda(z)) - {F}_{s-}(z) \right]
.$$
Then one obtains (see \cite{{berger2020growth}}) on any fixed time interval $[0,T]$ that $$F_t=\tilde{F}_T\circ\left[\lim_{s\searrow t}\tilde{F}_{T-s}\right]^{-1}.$$
Recall also
\begin{proposition}\cite[Theorem 3.1+Theorem 4.3]{berger2020growth}
The limit in \eqref{eq:back_around_middle} exists in the topology of mean square convergence on compact subsets of $\BH$ and compact time intervals. The process is well defined, in the sense that there is a unique process in law satisfying the conditions of Definition \ref{def:BSHL}.

\end{proposition}

\section{Definition of the CHL$^N$ Process}
	First we want to define a conformal slit mapping $\slitN_x(z)$, which attaches a slit over a point $x$ in the boundary of a cylinder of radius $N$, where $\delta$ is a parameter of the slit length.
	For that purpose we define the basic functions that form the cylinder slit map. Define the cylinder
	$${\BT^N=\{z\in\BC: \ima(z)>0\}}/{z\sim z+2N\pi},$$
	the complement of the unit disk
	$
	\BD_0=\{z\in \BC: |z|>1\},
	$
	and the upper half plane
	$
	\BH=\{z\in\BC: \ima(z)>0\}
	.$
	We denote the following functions: 
	\begin{align*}
		&f_N:\BT^N\to \BD_0     &&f_N(z)=e^{-i\frac{z}{N}},\\
		&f_N^{-1}:\BD_0\to \BT^N     &&f_N^{-1}(z)=iN\log(z),\\
		&g:\BD_0\to\BH &&g(z)=\frac{i(z-1)}{z+1},\\
		&g^{-1}:\BH\to\BD_0 &&g^{-1}(z)=\frac{i+z}{i-z},\\
		&\slit^\delta:\BH\to \BH\setminus [0,i\delta] &&\slit^\delta (z)=\sqrt{z^2(1-\delta^2)-\delta^2},\\
	\end{align*}
	Note that one can also use the complex tangent $$g\circ f_N(z)=\tan\left(\frac{z}{2N}\right).$$
	
	For two complex numbers $z,w\in\BC$, denote $[z,w]=\{tz+(1-t)w:t\in[0,1]\}$. Next we define the cylindrical slit map $\slitN:\BT^N\to \BT^N\setminus [0,i\lambda(\delta)]$ by taking
	$$
	\slitN(z)=f_N^{-1}\circ g^{-1}\circ \slit^{\delta}\circ g\circ f_N(z)
	.$$
	\begin{figure}[H]
		\include{conformal}
		\caption{The basic cylinder slit map $\slitN(z):\BT^N\to\BT^N\setminus[0,i\lambda]$.}		\label{basic_slit_map}
	\end{figure}
	Note that in \cite{berger2020growth} the slit map on $\BH$ is defined as $\sqrt{z^2-\delta^2}$. Here we need $i$ to be a fixed point for $\varphi^\delta(\cdot)$, otherwise infinity on the cylinder would not be mapped to infinity on the cylinder after adding the slit. However, we will see that the correction vanishes in the limit $N\to\infty$. 
	
		Also denote the rotation map $r_x:\BC\to\BC$, $r_x(z)=e^{ix/N}z$. The slit map attaching a slit of length $\lambda$ above $x\in\BT^N$, can be written as,
	\begin{equation}\label{eq:slitmapx}
		\slitN_x(z)=f_N^{-1}\circ r_x^{-1}\circ g^{-1}\circ \slit^{\delta}\circ g\circ r_x\circ f_N(z)
		,\end{equation}
	where the appropriate scaling parameter $\delta=\scal$, can be calculated by following the $0$ point, namely $\slitN(0)=i\lambda$ (see figure \ref{basic_slit_map}). Thus,
	
	\begin{equation}
		0\overset{f}{\to}1\overset{g}{\to}0\overset{\slit^{\delta}}{\to} i\delta\overset{g^{-1}}{\to}\frac{1+\delta}{1-\delta }\overset{f^{-1}}{\to}iN\log\left(1+\frac{2\delta}{1-\delta}\right)=i\lambda.
	\end{equation}
	We obtain
	\begin{equation}\label{eq:deltan}
	\scal=1-\frac{2}{1+e^{\frac{\lambda}{N}}}.
	\end{equation}

	For large $N$, or small $\lambda$, $\scal\approx \frac{\lambda}{2N}$. Note, that $\slitN_0=\slitN$. For ease of notation, since $\slitN(\cdot)$ is a non injective periodic function on $\BH$, we can also think of $\slitN(\cdot)$ as a conformal map from $\BH$ to $\BH\setminus \bigcup_{i\in\BZ}[i2\pi,i2\pi+i\lambda]$, mapping $x+iy$ for any $x\notin[-N\pi,N\pi]$, to $x+\slitN(x+iy)$. Thus we will make constant use of the identity for $z\in\BH$
\begin{equation}
\slitN_x(z)=\Re(z)+\slitN(z-x)-(\Re(z- x) \mod \pi N).
\end{equation}
And thus
\begin{equation}\label{eq:periodic}
\slitN_x(z)-z=\slitN(z-x)-(\Re(z-x) \mod \pi N)-i\Im(z).
\end{equation}

Now, let us define SHL$(0)$ in the cylinder. For the purpose of comparing CHL$^N$ and SHL$(0)$, we define the cylindrical process in continuous time, though, since the cylinder is finite one can define it also in discrete time. 
\begin{definition}\label{def:process}
	{Consider $\prob$, an intensity 1 Poisson point process on $\BH$.
		Let $A_{t,N}$ be the set of points distributed due to $\prob$ on $[-\pi N,\pi N)\times (0,t]$.
	}
	Almost surely, $\prob$ has finitely many points in any compact set. Hence we can write
	$ A_{t,\pi N}=\{(x_1,t_1),(x_2,t_2),\dots (x_n,t_n)\}$ s.t. $0\hm<t_1\hm<t_2<\dots \hm<t_n\leq t $ and $\forall i, -\pi N<x_i\leq\pi N$. CHL$^N$ is the c\'adl\'ag function $\mathfs{A}_t^{N,\delta}(z)$ s.t.
	\[ 
\mathfs{A}_t^{N,\delta}(z)=z+\sum_{(s,x)\in A_{t,N}}{[\mathfs{A}_{s^-}^{N,\delta}(\slitN_x(z))-\mathfs{A}_{s^-}^{N,\delta}(z)]} 
,\]
	which also can written as
	$$
	\mathfs{A}_s^{N,\delta}(z)=
	\begin{cases} 
		z & 0\leq s< t_1 \\
		\slitN_{x_1}(z) & t_1 \leq s < t_{2} \\
		\slitN_{x_1}\circ \slitN_{x_2}(z) & t_2 \leq s < t_{3} \\
		\vdots\\
	 	\slitN_{x_1} \circ \slitN_{x_2} \circ \dots \circ \slitN_{x_n}(z)&  t_n \leq s \leq t \\
	\end{cases}
	.$$
The backward CHL$^N$ process is defined by
\bae\label{eq:backcylprocess}
\tilde{\mathfs{A}}_t^{N,\delta}(z)=z+\sum_{(s,x)\in A_{t,N}}{[\slitN_x(\tilde{\mathfs{A}}_{s^-}^{N,\delta}(z))-\tilde{\mathfs{A}}_{s^-}^{N,\delta}(z)]} 
.\eae
\end{definition}
 Note that CHL$^N$ up to scaling and time change can be obtained by 
 \begin{remark}\label{rem:classicalHL_connection}
 Note that CHL$^N$ up to scaling and time change can be obtained by the classical HL$(0)$ using the construction given in \cite{norris2012hastings}. Let $\Psi_n^\delta=\psi_1^\delta \circ \psi_2^\delta \circ \dots \psi_n^\delta$, be the $\HL(0)$, where $\psi_i$ is a conformal map from $\BD_0 $ to a subset of $\BD_0 $, which corresponds to the i'th slit of size $\delta$. Since a half plane slit of length $\delta$ corresponds to a slit of size  $\frac{2\delta}{1-\delta}$ on the complement of disk (see Figure \ref{basic_slit_map}), we obtain that
	\begin{equation*}
		\begin{aligned}
			\tilde{\mathfs{A}}_n^{N,\delta}(z)
			& =f_N^{-1} \circ \psi_1^{\frac{2\delta}{1-\delta}} \circ f_N \circ f_N^{-1} \circ \psi_2^{\frac{2\delta}{1-\delta}}\circ f_N \circ \dots f_N^{-1} \circ \psi_n^{\frac{2\delta}{1-\delta}} \circ f_N (z) =\\
			& = f_N^{-1}\circ \psi_1\circ \psi_2 \circ \dots \psi_n \circ f_N (z) =\\
			& = f_N^{-1}\circ \Psi_n^{\frac{2\delta}{1-\delta}} \circ f_N(z). \\
		\end{aligned}	
	\end{equation*}

Thus the correct scaling to obtain CHL$^N$ from the HL$(0)$ for large $N$ by \eqref{eq:deltan}, is taking $\delta^N\approx \frac{\lambda}{2N}$ and time scaling $n=\lceil 2\pi Nt\rceil$ i.e. the expected number of slits in $A_{t,N}$.
\end{remark} 
\section{Results and Discussion}
%
%
	
\subsection{Main result}	
The main result of this paper states that under the appropriate scaling which fixes the slit size to $\lambda$, as $N\to\infty$ the CHL$^N$ process converges to the SHL$(0)$ process in the topology of mean square convergence on compact subsets of $\BH$ and compact time intervals. We use the natural coupling in which both processes use the same Poisson point process $A_t$, denoted $\prob$ with expectation $\ev$. 
\begin{theorem}\label{thm:main}
For any compact $[0,t]$ and $K\subset \BH$
$$
\ev\left[\sup_{s\leq t}\sup_{z\in K}\left|\mathfs{A}_s^{N,\delta}(z)-F_s^\lambda (z)\right|^2\right] \to 0
,$$
as $N \to \infty$ and $\delta=\delta(N, \lambda)$.
\end{theorem}
\begin{remark}
An interesting interpretation of this result is that the classical HL$(0)$ model in the small particle limit $\delta\to0$  grown up to time $t\delta^{-1}$ appropriately scaled around a point on $\BH$ and normalized such that particles are of size $1$, locally converges to SHL$(0)$ (see Remark \ref{rem:classicalHL} for an explanation of the connection between the models). 
\end{remark}
\subsection{Ergodicity} In  \cite[Proposition 4.7]{berger2020growth} it is proved that the SHL$(0)$ is ergodic with respect to real shifts. In our cylindrical geometry it is actually easier to prove an equivalent statement. 
\begin{proposition}
For any $t>0$, $\mathfs{A}_t^{N,\delta}(z)$ is ergodic with respect to the cylinder shift $\eta_y(z):=f_N^{-1} \circ r_y\circ f_N(z)$.
\end{proposition}
\begin{proof}
We prove ergodicity by showing that the CHL$^N$ is a factor of the ergodic intensity 1 Poisson point process. Thus it is enough to prove that the function that takes the Poisson point process on $\BT^N$ and maps it to $\mathfs{A}_t^{N,\delta}(z)$ commutes with $\eta_y(\cdot)$. For any $n\in\BZ$ and $y, x_1,\cdots,x_n\in [-N\pi,N\pi)$,
\bae\label{eq:ergodic}
&\eta^{-1}_y\circ\slitN_{x_1} \circ \slitN_{x_2} \circ \dots \circ \slitN_{x_n}(\eta_y(z))\\
&=\eta^{-1}_y\circ f_N^{-1}\circ r_{x_1}^{-1}\circ g^{-1}\circ \slit^{\delta}\circ g\circ r_{x_1}\circ \cdots \circ r_{x_n}^{-1}\circ g^{-1}\circ \slit^{\delta}\circ g\circ r_{x_n}\circ f_N(\eta_y(z))\\
&=f_N^{-1}\circ r^{-1}_y\circ r_{x_1}^{-1}\circ  g^{-1}\circ \slit^{\delta}\circ g\circ r_{x_1}\circ r_y\circ \cdots \circ r^{-1}_y \circ r_{x_n}^{-1}\circ g^{-1}\circ \slit^{\delta}\circ g\circ r_{x_n}\circ r_y \circ f_N(z)\\
&=f_N^{-1}\circ r_{x_1+y}^{-1}\circ  g^{-1}\circ \slit^{\delta}\circ g\circ r_{x_1+y}\circ \cdots \circ r_{x_n+y}^{-1}\circ g^{-1}\circ \slit^{\delta}\circ g\circ r_{x_n+y}\circ f_N(z)\\
&=\slitN_{x_1+y} \circ \slitN_{x_2+y} \circ \dots \circ \slitN_{x_n+y}(z)
,\eae
where in the second equality we composed $n-1$ times with the identity function $r_y\circ r^{-1}_y(\cdot)$.
The last line of \eqref{eq:ergodic} is exactly the CHL$^N$ process constructed with the points $(\eta_y(x_1),t_1),\ldots, (\eta_y(x_n),t_n)$.

\end{proof}

\section{Cylinder Slit Function Estimates}
Before proving Theorem \ref{thm:main}, we state in this section and prove in the appendices some estimates on the deterministic slit maps.
\begin{lemma}\label{lem:conv_slit_fun}
		For every  $z\in\mathbb{T}^N$, there is a $C(z)<\infty$ such that for any $N$ large enough
		$$
		\left|\slitN_0(z)-\sqrt{z^2-1}\right| \leq \frac{C(z)}{N}
		,$$
		whenever $\delta=\delta(N,\lambda)$.
\end{lemma}
	
Proof of the Lemma \ref{lem:conv_slit_fun} appears in Appendix \ref{AppA}.

	The following lemma is the cylinder versions of \cite[Lemma 3.3]{berger2020growth}. The proof of the Lemma appears in Appendix \ref{appendix:deterministic}.
	
	\begin{lemma}\label{lem:3.3}
	There exists a constant $C>0$ such that for every $z\in\mathbb{T}^N$
	\begin{align}
		\int_{-N\pi}^{N\pi}|\slitN_x(z)-z|^2dx<C \label{eq:lem_3.3_1} \\
		\int_{-N\pi}^{N\pi}\left|\frac{d}{dz}\slitN_x(z)-1\right|^2dx<C \label{eq:lem_3.3_2}
	\end{align}
	\end{lemma}
	Equation \eqref{eq:tailest} in the proof of Lemma \ref{lem:3.3} also gives us 
	\begin{lemma}\label{lem:uniformtail}
For every $\epsilon>0$ and any $z\in\mathbb{T}^N$, there exists an $N>0$ such that for any $M>N$,
$$
\int_{N\pi}^{M\pi}|\slitM_x(z)-z|^2dx<\epsilon
.$$
\end{lemma}

Next we control the slit map at infinity.

\begin{lemma}\label{lem:rate}
For any $N>0$, as $Im(z)\to\infty$, $z\in \mathbb{T}^N$,  
$$
\slitN_0 (z)=z-iN\log (1-\delta^2)+iN\frac{\delta^2}{e^{-iz\backslash N}}+o\left(e^{-|z|}\right)
.$$
\end{lemma}
Proof of the Lemma \ref{lem:rate} appears in Appendix \ref{AppA}. Using Lemma \ref{lem:rate}, one can calculate the average shift of the slit map, which in turn will provide us with growth rates of the process.
\begin{lemma}\label{lem:pathpartition}
For any $z\in \BT^N$, 
$$
\int_{-N\pi}^{N\pi}\left[\slitN_x\left(z\right)-z\right]dx=-i2N^2\pi\log(1-\delta^2)
.$$
\end{lemma}
\begin{proof}
Let $y=\ima(z)$, by \eqref{eq:periodic}
	\begin{equation}
	\begin{aligned}
	\int_{-N\pi}^{N\pi}\left[\slitN_x\left(z\right)-z\right]dx=\int_{-N\pi}^{N\pi}\left[\slitN\left(x+i y\right)-(x+i y)\right]dx
	\end{aligned}
	\end{equation}
	Fix $z \in \BT^N$ and consider the contour defined by the four curves $\gamma_i$, $i=1,2,3,4$ where
\begin{align*}
\gamma_1(t) &= z-t &\mbox{ for } &t \in [-N\pi,N\pi] \\
\gamma_2(t) &= z-N\pi + it &\mbox{ for } &t \in [0,R] \\
\gamma_3(t) &= z+t + iR &\mbox{ for } &t \in  [-N\pi,N\pi] \\ 
\gamma_4(t) &= z+N\pi +iR- it &\mbox{ for } &t \in [0,R],
\end{align*}
for some $N\ll R$ (see Figure \ref{gammas}).

\begin{figure}[h!]
	\begin{tikzpicture}[]
	\tikzstyle{bluecirc}=[circle,
draw=black,fill=magenta, thin,inner sep=0pt,minimum size=3mm]
	\node (c) at (-3,-0) [bluecirc] {};
	\node (cu) at (-3,2) [bluecirc] {};
	\node (d) at (3,0) [bluecirc] {};
	\node (du) at (3,2) [bluecirc] {};
	\draw [<-] (d)--(du);
	\draw [<-] (du)--(cu);
	\draw[<-]  (cu)--(c);
	\draw [<-] (c) -- (d);
	\node (zu) at (0,0) [bluecirc] {};
	\node (z) at (0,-0.3) {\footnotesize{$z$}};
	\node (a) at (-3,-0.3) {\footnotesize{$z-N\pi$}};
	\node (au) at (-4.3,2) {\footnotesize{$z-N\pi+iR$}};
	\node (b) at (3,-0.3) {\footnotesize{$z+N\pi$}};
	\node (bu) at (4.3,2) {\footnotesize{$z+N\pi+iR$}};
	\node (gam1) at (0,0.3) {\footnotesize{$\gamma_1$}};
	\node (gam2) at (-3.5,1) {\footnotesize{$\gamma_2$}};
	\node (gam4) at (3.2,1) {\footnotesize{$\gamma_4$}};
	\node (gam1) at (0,2.2) {\footnotesize{$\gamma_3$}};
	\end{tikzpicture}
	\caption{Illustration of the curves.}
	\label{gammas}
\end{figure}
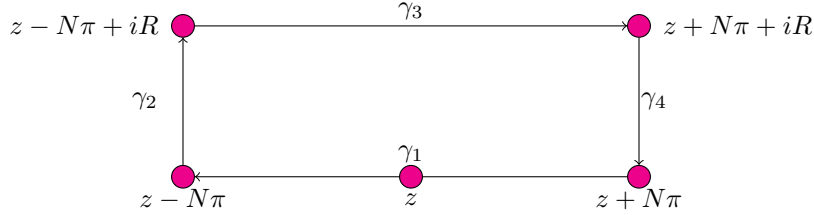

Since the function $g(w) = \slitN(w)-w$ is analytic within the contour, and by symmetry of the cylinder,
\begin{align*}
\int_{-N\pi}^{N\pi}\left[\slitN_x\left(z\right)-z\right]dx &= \int_{-N\pi}^{N\pi}\left[\slitN\left(x+i y\right)-(x+i y)\right]dx \\
&=-\oint_{\gamma_1} g(w) dw =  \left ( \oint_{\gamma_2} + \oint_{\gamma_3} + \oint_{\gamma_4}\right ) g(w) dw.
\end{align*}
On $\BT^N$ we have that $\oint_{\gamma_2} g(w)dw=-\oint_{\gamma_4}g(w)dw$, and thus
\begin{equation}\label{eq:gamma3int}
\int_{-N\pi}^{N\pi}\left[\slitN_x\left(z\right)-z\right]dx=\int_{-N\pi}^{N\pi}\left(\slitN(x+iy+iR)-(x+iy+iR)\right)dx
\end{equation}
Since the LHS of \eqref{eq:gamma3int} does not depend on $R$, we can take the limit $R\to\infty$ and by Lemma \ref{lem:rate} we obtain that 
\begin{equation}\label{eq:Rlimit}
\int_{-N\pi}^{N\pi}\left[\slitN_x\left(z\right)-z\right]dx=-i2N^2\pi\log(1-\delta^2)
\end{equation}
	
\end{proof}	


%
\section{Doob Decomposition and Growth Rate}
In this section we show that $\tilde{\mathfs{A}}_t^{N,\delta}(z)$ is a Martingale in $t$.
	

	\begin{lemma}\label{lem:doob}
	There is a zero-mean martingale $\mathfs{M}^{N,\delta}_t$ such that
		\bae
		\tilde{\mathfs{A}}_t^{N,\delta}(z)=z +\mathfs{M}^{N,\delta}_t(z)-i 2\pi t N^2\log(1-\delta^2)
		\eae
	\end{lemma}
	\begin{proof}
Recall Definition \ref{def:process}. By Doob decomposition theorem one can write $\tilde{\mathfs{A}}_t^{N,\delta}(z)=z+\mathfs{M}_t^{N,\delta}(z)+\mathfs{D}_t^{N,\delta}(z)$, where $\mathfs{M}_t^{N,\delta}$ is a mean-zero martingale, and $\mathfs{D}_t^{N,\delta}$ is predictable and given by
			\begin{equation*}
				\begin{aligned}
					\mathfs{D}_t^{N,\delta}(z)
					&=\int_{0}^{t}\int_{-N\pi}^{N\pi}\left[\slitN_x\left(\tilde{\mathfs{A}}_{s^-}^{N,\delta}(z)\right)-\tilde{\mathfs{A}}_{s^-}^{N,\delta}(z)\right]dxds.
				\end{aligned}
				\end{equation*}
			We conclude the result by replacing the inner integral using Lemma \ref{lem:pathpartition}, that is
			\begin{equation*}
				\begin{aligned}
					\mathfs{D}_t^{N,\delta}(z)=-i2\pi N^2 t \log\left(1-\delta^2\right)
				\end{aligned}
				.\end{equation*}
	\end{proof}


By Lemma \ref{lem:doob} and \eqref{eq:deltan}, we immediately obtain the average growth rate of the CHL$^N$ process as $N\to\infty$, which corresponds to the growth of the SHL$(0)$ in \cite[Lemma 5.1]{berger2020growth} .
	
\begin{corollary}\label{cor:errlimit}
\bae	
\lim_{N\to\infty}\ev[\tilde{\mathfs{A}}_t^{N,\delta(N,\lambda)}(z)]&=z+\lim_{N\to\infty}\mathfs{D}_t^{N,\delta}(z)=z+i \frac{\pi\lambda^2 t}{2}\\
&=z+\lim_{N\to\infty}-i2\pi N^2 t \log\left(1-\delta(N,\lambda)^2\right)\\
&=z+i \frac{\pi\lambda^2 t}{2}
.\eae
\end{corollary}


\begin{remark}\label{rem:classicalHL}
Note that an alternative approach to seeing the growth rate obtained in Corollary \ref{cor:errlimit}, is based on the Laurent expansion given in \cite{norris2012hastings}. By Remark \ref{rem:classicalHL_connection}, we have
	$$
	\tilde{\mathfs{A}}_n^{N,\delta}(z)=f_N^{-1}\circ \Psi_n^{{\frac{2\delta}{1-\delta}}}\circ f_N(z)
	.$$
Based on \cite{norris2012hastings} we can write $\Psi_n^{{\frac{2\delta}{1-\delta}}}(z)=e^{cn}z+O(1)$, as $z\to\infty$.

%
%
%
	We get
	$$
	\tilde{\mathfs{A}}_n^{N,\delta}(z)=iN\log\left(e^{cn-iz/N}+O(1)\right)=z+icnN+o(1)\text{, as }z\to\infty
	.$$
	By \cite[Footnote$^5$]{norris2012hastings} 
	$$c\left({\frac{2\delta}{1-\delta}}\right)=-\log\left(1-\frac{{\frac{2\delta}{1-\delta}}^2}{(2+{\frac{2\delta}{1-\delta}})^2}\right)\approx {\delta^2}, \text{ as }\delta\to 0.$$


By \eqref{eq:deltan}, $\delta^N\approx \frac{\lambda}{2N}$ for a large enough $N$. We obtain by taking the time scaling $n=\lceil 2\pi Nt\rceil$, that
	$$
	\lim_{N\to\infty}\tilde{\mathfs{A}}_t^{N,\delta(N,\lambda)}(z)=z+i \frac{\pi\lambda^2 t}{2}+o(1),\text{ as }z\to\infty
	.$$
\end{remark}

The final Lemma in this section is a standard compactness and conformal distortion argument, that simplifies the topology in which we test the convergence in the main Theorem. Remember we denoted $\mathcal X$ to be the space of c\'adl\'ag functions from $[0,\infty)$ to $\mathcal Z$, with the Borel $\sigma$-algebra induced by the Skorohod topology.
\begin{lemma}\label{lem:compactness}
Let $\{\mathfs{Y}_s\}_{s\ge 0}\in\mathcal{X}$, then for any compact $K\subset\BH$, there exists $N\in\BN$, $z_1,\ldots,z_N\in K$ such that
$$
\ev\left[\sup_{s\le t}\sup_{z\in K}|\mathfs{Y}_s(z)|^2\right]\le 2\sum_{i=1}^N \ev\left[\sup_{s\le t}|\mathfs{Y}_s(z_i)|^2\right]
+32\sum_{i=1}^N \ev\left[\sup_{s\le t}\left|\frac{\partial}{\partial z}\mathfs{Y}_s(z_i)\right|^2\right]
.$$
\end{lemma}
\begin{proof}
For any $z_0\in\BH$, define $g_t:\{|z|<1\}\to\BC$, by 
$$
g_t(z)=\frac{\mathfs{Y}_t(z_0+rz)}{r\mathfs{Y}'_t(z_0)}
,$$
where $r\le \ima z_0$. Then $g_t$ is a conformal map with $g'(0)=1$. By conformal distortion \cite[Theorem 4.5]{garnett2005harmonic} 
$$
|g_t'(z)|\le\frac{1+|z|}{(1-|z|)^3}
.$$
Thus for all $z\in B(z_0,r/2)$, one can write $z=z_0+r\xi$, with $|\xi|\le1/2$. We get that
$$
|\mathfs{Y}_t'(z)|\le 16 |\mathfs{Y}_t'(z_0)|
.$$
For any compact set $K\subset\BH$, one can find an $N\in\BN$, $z_1,\ldots,z_N\in K$ and $0<r_i<\ima z_1/2\wedge 1$, such that $$
K\subset \bigcup_{i=1}^N B(z_i,r_i).
$$
For any $z\in K$ there is some $i\in\{1,\ldots, N\}$ such that $z\in B(z_i,r_i)$, and thus
$$
|\mathfs{Y}_t(z)|\le |\mathfs{Y}_t(z_i)|+|z-z_i|\sup_{\xi\in B(z_i,r_i)}\left|\mathfs{Y}'_t(\xi)\right |\le |\mathfs{Y}_t(z_i)|+16\left|\mathfs{Y}_t'(z_i)\right|
.$$
Finally we get by using $(|a|+|b|)^2\le 2|a|^2+2|b|^2$
$$
\ev\left[\sup_{s\le t}\sup_{z\in K}|\mathfs{Y}_s(z)|^2\right]\le 2\sum_{i=1}^N \ev\left[\sup_{s\le t}|\mathfs{Y}_s(z_i)|^2\right]
+32\sum_{i=1}^N \ev\left[\sup_{s\le t}\left|\frac{\partial}{\partial z}\mathfs{Y}_s(z_i)\right|^2\right]
.$$
\end{proof}	


\section{Proof of Convergence of CHL$^N$ Process to SHL$(0)$}

	For any $N\in\BN$ denote by $T_N=\inf\{t\ge 0: |\tilde{\mathfs{A}}_t^{N,\delta}(z)|>\frac{N^{1/2}}{4}\}$,  and let $T=T_N\wedge T_M$.
	\begin{corollary}\label{cor:stoptime}
	$\prob\left[ T<t \right]\le\frac{C(t,z)}{N\wedge M}.$
	\end{corollary}
	\begin{proof}
	By  Markov's inequality and then $(|z|+|w|)^2\le 2|z|^2+2|w|^2$,
$$
\prob\left[ T_N<t \right]\le\frac{16\ev\left[\sup_{s\le t}|\tilde{\mathfs{A}}_s^{N,\delta}(z)|^2\right]}{N}\le\frac{32\ev\left[\sup_{s\le t}|\tilde{\mathfs{A}}_s^{N,\delta}(z)-z|^2\right]+32|z|^2}{N}
.$$
	By Lemma \ref{lem:doob}, there is a constant $C(t)$ such that
	\bae
	\ev\left[\sup_{s\le t}\Big|\tilde{\mathfs{A}}_s^{N,\delta}(z)-z\Big|^2\right]&\le 2\ev\left[\sup_{s\le t}\Big|\mathfs{M}_s^{N,\delta}(z)\Big|^2\right]+C(t)\\
	&\le 2\ev\left[\int_0^t\int_{-N\pi}^{N\pi}\Big|\slitN_x(\tilde{\mathfs{A}}_s^{N,\delta}(z))-\tilde{\mathfs{A}}_s^{N,\delta}(z)\Big|^2dxds\right]+C(t)
	,\eae
	where the second inequality is derived from Doob's martingale inequality.
	By Lemma \ref{lem:3.3} the right summand is bounded by a constant. Finally we finish by noting that 
	$$
	\prob\left[ T<t \right]\le\prob\left[ T_N<t \right]+\prob\left[ T_M<t \right]
	.$$
	\end{proof}
	
The following Theorem states that the sequence $\{\tilde{\mathfs{A}}_s^{N,\delta}(z)\}_N$ is Cauchy in the topology of mean square convergence on compacts. 
\begin{theorem}\label{lem:cauchyA}
For any compact $[0,t]$ and $K\subset \BH$ and any $\ep>0$, there exists an $\tilde{N}\in\BN$ such that for any $M>N>\tilde{N}$,
$$
\ev\left[\sup_{s\leq t}\sup_{z\in K}\left|\tilde{\mathfs{A}}_s^{N,\delta}(z)-\tilde{\mathfs{A}}_s^{M,\delta}(z)\right|^2\right]<\ep
.$$

\end{theorem}

\begin{proof}
By Lemma \ref{lem:doob} one can decompose
		\begin{equation*}
			\tilde{\mathfs{A}}_s^{N,\delta}-\tilde{\mathfs{A}}_s^{M,\delta}=\mathfs{M}_s^{N,\delta}-\mathfs{M}_s^{M,\delta}+\mathfs{D}_s^{N,\delta}-\mathfs{D}_s^{M,\delta}
		\end{equation*}
		and so by triangle inequality we get
		\begin{equation}\label{eq:triangleDoob}
			\begin{aligned}
				\ev\left[\sup_{s\leq t}\left|\tilde{\mathfs{A}}_s^{N,\delta}(z)-\tilde{\mathfs{A}}_s^{M,\delta}(z)\right|^2\right]&\leq
				\ev\left[\sup_{s\leq t}\left|\mathfs{M}_s^{N,\delta}(z)-\mathfs{M}_s^{M,\delta}(z)\right|^2\right] \\
				&+\ev\left[\sup_{s\leq t}\left|\mathfs{D}_s^{N,\delta}(z)-\mathfs{D}_s^{M,\delta}(z)\right|^2\right].
			\end{aligned}
		\end{equation}	
		Note that by Corollary \ref{cor:errlimit}, the drift part $\mathfs{D}_s^{N,\delta}-\mathfs{D}_s^{M,\delta}$ is deterministic and is of order smaller than $\frac{c}{\tilde{N}^2}$.

By Lemma \ref{lem:doob} it is enough to prove for any $M>N>\tilde{N}$,
\begin{equation}
\ev\left[\sup_{s\leq t}\sup_{z\in K}\left|{\mathfs{M}}_s^{N,\delta}(z)-{\mathfs{M}}_s^{M,\delta}(z)\right|^2\right]<\ep,
\end{equation}
where $\mathfs{M}_s^{N,\delta}(z)$ is the martingale part of $\tilde{\mathfs{A}}_s^{N,\delta}(z)$ one obtains from the Doob decomposition theorem.
By Lemma \ref{lem:compactness}, it is enough to prove for any $z\in\BH$ that
\begin{equation}
\ev\left[\sup_{s\leq t}\left|\mathfs{M}_s^{N,\delta}(z)-\mathfs{M}_s^{M,\delta}(z)\right|^2\right]<\ep
,\end{equation}
and
\begin{equation}\label{eq:derconvergence}
\ev\left[\sup_{s\leq t}\left|\frac{\partial}{\partial z}\left(\mathfs{M}_s^{N,\delta}(z)-\mathfs{M}_s^{M,\delta}(z)\right)\right|^2\right]<\ep
.\end{equation}
Indeed
{\small \bae\label{eq:martingalediff}
&\ev\left[\sup_{s\leq t}\left|\mathfs{M}_s^{N,\delta}(z)-\mathfs{M}_s^{M,\delta}(z)\right|^2\right]\\
&\le 16\ev\left[\int_0^t\int_{N\pi}^{M\pi}|\slitM_x(\tilde{\mathfs{A}}_s^{M,\delta}(z))-\tilde{\mathfs{A}}_s^{M,\delta}(z)|^2dxds\right]\\
&+8\ev\left[\int_0^t\int_{-N\pi}^{N\pi}|(\slitN_x-id)(\tilde{\mathfs{A}}_s^{N,\delta}(z))-(\slitM_x-id)(\tilde{\mathfs{A}}_s^{M,\delta}(z))|^2dxds\right]
\eae }
For $s<t$
\bae\label{eq:indexpbound}
\ev\left[\int_{N\pi}^{M\pi}|\slitM_x(\tilde{\mathfs{A}}_s^{M,\delta}(z))-\tilde{\mathfs{A}}_s^{M,\delta}(z)|^2 dx\right] &\le\ev\left[\int_{N\pi}^{M\pi}|\slitM_x(\tilde{\mathfs{A}}_s^{M,\delta}(z))-\tilde{\mathfs{A}}_s^{M,\delta}(z)|^2\ind_{\{T>t\}}dx\right]\\
&+\ev\left[\int_{-M\pi}^{M\pi}|\slitM_x(\tilde{\mathfs{A}}_s^{M,\delta}(z))-\tilde{\mathfs{A}}_s^{M,\delta}(z)|^2\ind_{\{T<t\}}dx\right]\\
&\le \ep+ \frac{C_1(t,z)}{M},
\eae
where the last inequality is due to Lemma \ref{lem:uniformtail}, \ref{lem:pathpartition} and Corollary \ref{cor:stoptime} by taking $\tilde{N}$ large enough.

Next we attend to the second summand in the RHS of \eqref{eq:martingalediff} and  rewrite 
$$
\ev\left[\int_0^t\int_{-N\pi}^{N\pi}\Big|(\slitN_x-\text{id})(\tilde{\mathfs{A}}_s^{N,\delta}(z))-(\slitN_x-\text{id})(\tilde{\mathfs{A}}_s^{M,\delta}(z))+(\slitN_x-\slitM_x)(\tilde{\mathfs{A}}_s^{M,\delta}(z))\Big|^2dxds\right]
.$$
Next,
\bae
&\lim_{\tilde{N}\to\infty}\ev\left[\int_0^t\int_{-N\pi}^{N\pi}\Big|(\slitN_x-\slitM_x)(\tilde{\mathfs{A}}_s^{M,\delta}(z))\Big|^2dxds\right]\le\\
&\lim_{\tilde{N}\to\infty}2t\ev\left[\int_{0}^{\xi}\Big|(\slitN_x-\slitM_x)(0)\Big|^2dx\right]+\lim_{\tilde{N}\to\infty}8t\ev\left[\int_{\xi}^{N\pi}\Big|(\slitN_x-\text{id})(0)\Big|^2dx\right]
.\eae
By Equation \eqref{eq:tailest} for any $\tilde{N}$ large enough $\ev\left[\int_{\xi}^{N\pi}\Big|(\slitN_x-\text{id})(0)\Big|^2dx\right]\le \frac{c}{\xi}$, thus we can choose $\xi$ large enough such that $\frac{c}{\xi}<\epsilon$. By Lemma \ref{lem:conv_slit_fun} for the chosen $\xi$, $$\lim_{\tilde{N}\to\infty}2t\ev\left[\int_{0}^{\xi}\Big|(\slitN_x-\slitM_x)(0)\Big|^2dx\right]=0.$$
Lastly,
$$
|(\slitN_x-id)(\tilde{\mathfs{A}}_s^{N,\delta}(z))-(\slitN_x-id)(\tilde{\mathfs{A}}_s^{M,\delta}(z))|\le|\tilde{\mathfs{A}}_s^{N,\delta}(z)-\tilde{\mathfs{A}}_s^{M,\delta}(z)|\int_0^1|\slitN_x(\gamma(u))'-1|du,
$$
where $\gamma(u)$ is a curve linearly interpolating between $\tilde{\mathfs{A}}_s^{M,\delta}(z)$ and $\tilde{\mathfs{A}}_s^{N,\delta}(z)$. 

By  \eqref{eq:martingalediff}, \eqref{eq:indexpbound}, Lemma \ref{lem:3.3}, Jensen and Fubini 
$$
\ev\left[\sup_{s\leq t}\left|\mathfs{M}_s^{N,\delta}(z)-\mathfs{M}_s^{M,\delta}(z)\right|^2\right]\le C(t,z)\left(\int_0^tg(s)ds+\frac{1}{\tilde{N}}\right)
,$$
where $g(t)=\ev\left[\sup_{s\leq t}\left|\tilde{\mathfs{A}}_s^{N,\delta}(z)-\tilde{\mathfs{A}}_s^{M,\delta}(z)\right|^2\right]$.
By Gr{\"o}nwall's Lemma,
$$
		\ev\left[\sup_{s\leq t}\left|\tilde{\mathfs{A}}_s^{N,\delta}(z)-\tilde{\mathfs{A}}_s^{M,\delta}(z)\right|^2\right]=g(t)\leq\frac{C_3(t,z)}{\tilde{N}}e^{\int_{0}^{t} C_2(t,z)ds} = \frac{C_3(t,z)}{\tilde{N}}e^{t C_2(t,z)},
		$$
and the conclusion follows by taking $\tilde{N}$ large enough. 

Similar arguments yields the derivatives' convergence \eqref{eq:derconvergence}, by noting that for large enough $N$,
\bae
&\int_0^{N\pi} \left| \frac{d^2}{dz^2}\slitN_x(z)\right|^2dx=\\
&\int_0^{N\pi} \left| \frac{\lambda^4}{4^4N^6\left(2\frac{\lambda^2}{4N^2}+\cosh\left(\frac{-iz}{N}\right)-1\right)^2\left(1-\frac{\lambda^2}{4N^2}\tanh^2\left(\frac{z}{2N}\right)-\frac{\lambda^2}{4N^2}\right)}\right|dx,
\eae
which uniformly converges to zero as $N\to\infty$, by using the same estimates as \eqref{eq:blowpart} and \eqref{eq:tailpart} for any $y=\ima(z)>0$.
\end{proof}



\begin{proof}[Proof of Theorem \ref{thm:main}]
By equidistribution of the forward and backward SHL$(0)$ and CHL$^N$ for a fixed time $t$, it is enough to prove that the limit guarantied by Lemma \ref{lem:cauchyA}, denoted $\tilde{\mathfs{A}}^{\infty,\delta}_t$, is a c\'adl\'ag map satisfying conditions (1)-(4) in Definition \ref{def:BSHL}. (1) and (2) are immediate from Definition \ref{def:process}. To prove (3) it is enough to prove that for $0\leq s<t$ one has ${\tilde{\mathfs{A}}}^{N,\delta}_t\circ \left({\tilde{\mathfs{A}}^{N,\delta}_s}\right)^{-1}$ is measurable with respect to $\tilde \CF_{s,t} = \sigma\left(\tilde P|_{(s,t]\times[-N\pi,N\pi)}\right)$. We partition $A_{t,\pi N}$ from Definition \ref{def:process} into, $A_{s,\pi N}=\{(x_1,t_1),\dots ,(x_k,t_k)\}$ and $A_{t,\pi N}\setminus A_{s,\pi N}=\{(x_{k+1},t_{k+1}),\dots ,(x_n,t_n)\}$. Then,
\begin{equation*}
	\begin{aligned}
			{\tilde{\mathfs{A}}}^{N,\delta}_t\circ \left({\tilde{\mathfs{A}}^{N,\delta}_s}\right)^{-1}&=\slitN_{x_n}\circ\dots\circ\slitN_{x_{k+1}}\circ\slitN_{x_k}\circ\dots\circ\slitN_{x_1}\circ\left(\slitN_{x_1}\right)^{-1}\circ\dots\circ\left(\slitN_{x_k}\right)^{-1}\\
		&=\slitN_{x_n}\circ\dots\circ\slitN_{x_{k+1}}\\
	\end{aligned}
\end{equation*}
is a function of only points from $A_{t,\pi N}\setminus A_{s,\pi N}.$ Therefore it is measurable w.r.t. the sigma-algebra generated by Poisson points in $(s,t]\times[-N\pi,N\pi)$. 

We are left with proving (4). By Lemma \ref{lem:cauchyA} and \eqref{eq:backcylprocess}, we can take a sequence $N_k\to\infty$ such that 
\bae\label{eq:asconv}
{\tilde{\mathfs{A}}}^{\infty,\delta}_t=\lim_{k\to\infty} {\tilde{\mathfs{A}}}^{N_k,\delta}_t=z+\lim_{k\to\infty} \sum_{(s,x)\in A_{t,N_k}}{[\mathfs{S}^{N_k,\delta}_x(\tilde{\mathfs{A}}_{s^-}^{N_k,\delta}(z))-\tilde{\mathfs{A}}_{s^-}^{N_k,\delta}(z)]}\text{, a.s}
.\eae
Thus if $(s,x)\in A_{t}$ ,
\bae
{\tilde{\mathfs{A}}}^{\infty,\delta}_{s}-{\tilde{\mathfs{A}}}^{\infty,\delta}_{s-}=\lim_{k\to\infty} \mathfs{S}^{N_k,\delta}_x(\tilde{\mathfs{A}}_{s^-}^{N_k,\delta}(z))= \lim_{k\to\infty} \slit_x(\tilde{\mathfs{A}}_{s^-}^{N_k,\delta}(z))+\left(\mathfs{S}^{N_k,\delta}_x-\slit_x\right)(\tilde{\mathfs{A}}_{s^-}^{N_k,\delta}(z)).
\eae

By the continuity of $\slit_x$, Lemma \ref{lem:conv_slit_fun} and \eqref{eq:asconv}, we obtain 
\bae
{\tilde{\mathfs{A}}}^{\infty,\delta}_{s}-{\tilde{\mathfs{A}}}^{\infty,\delta}_{s-}= \slit_x(\tilde{\mathfs{A}}_{s^-}^{\infty,\delta}(z)) \text{, a.s}.
\eae
\end{proof}

\appendix \section{Proof of Lemma \ref{lem:3.3}}\label{appendix:deterministic}

\subsection{Proof of \eqref{eq:lem_3.3_1}}
\begin{proof}
Since \eqref{eq:lem_3.3_1} is maximized at $\ima(z)=0$, we will bound
\bae
\int_{-\pi N}^{\pi N} \left|\slitN_0(x)-x\right|^2 dx
\eae		
The function
\bae
g\circ f(x+iy)=\tan\left(\frac{x+iy}{2N}\right)=\frac{\sin(x/N)+i\sinh(y/N) }{cos(x/N)+\cosh(y/N)},
\eae
at the value $y=0$, is real valued, gets the value $0$ at $x=0$ and tends to infinity as $x\to\pi N$, with monotonically increasing derivative. There exists a $\xi$ such that for all $x>\xi$, 
$$\frac{\sin(x/N)}{1+\cos(x/N)}\ge \frac{x}{N}\ge  \frac{\lambda}{N}.$$
We use the scaling relation
\bae
\left| \varphi^{\delta}_0(x)-x \right|^2=\delta^2\left| \sqrt{x^2\left(\frac{1}{\delta^2}-1\right)-1}-\frac{x}{\delta} \right|^2
.\eae
By \cite[Appendix A]{berger2020growth}, we get that whenever $\frac{x}{\delta}<2$ or substituting for the value of $\delta_N$, $x\le\frac{\lambda}{N}$, 
\bae\label{eq:smallx}
\left| \varphi^{\delta}_0(x)-x \right|^2\le \frac{36\lambda^2}{N^2},
\eae
and whenever $x>\frac{\lambda}{N}$, 
\bae\label{eq:bigx}
\left| \varphi^{\delta}_0(x)-x \right|^2\le \left(\frac{\lambda}{2N}\right)^4\frac{1}{x^2}.
\eae
Thus by \eqref{eq:smallx} and \eqref{eq:bigx} there is a universal constant $c>0$, such that
\bae\label{eq:tailest}
&\int_{-\pi N}^{\pi N} \left|\slitN_0(x)-x\right|^2 dx\\
&=2\int_{0}^{\xi} \left|\slitN_0(x)-x\right|^2 dx+2\int_{\xi}^{\pi N} \left|\slitN_0(x)-x\right|^2 dx\\
&\le 2\int_{0}^{\xi} \left|2N\arctan\circ\tan(\frac{x+c}{2N})-x\right|^2 dx+2\int_{\xi}^{\pi N} \left|2N\arctan\circ\tan(\frac{x}{2N}+\frac{c}{Nx})-x\right|^2 dx\\
&\le 2\xi c^2+\frac{c^2}{\xi},
\eae	
where for the integral between $\xi$ and $\pi N$ we have used the monotonicity of the derivative of $g\circ f$.
\end{proof}
\subsection{Proof of \eqref{eq:lem_3.3_2}}\begin{proof}
		Without loss of generality suppose $z=iy$ for some $y>0$. 
		

		\item Next we prove \eqref{eq:lem_3.3_2}. For any $z\in \BT^N$,
		\bae
		\frac{d}{dz}\slitN(z)=\frac{1}{\sqrt{1-\frac{\lambda^2}{4N^2}-\frac{\lambda^2}{4N^2\tanh^2\left(\frac{z}{2N}\right)}}}.
		\eae
		We need to bound
		\bae
		&\int_0^{N\pi}\left|\frac{1}{\sqrt{1-\frac{\lambda^2}{4N^2}-\frac{\lambda^2}{4N^2\tanh^2\left(\frac{x+iy}{2N}\right)}}}-1\right|^2dx
		.\eae 
		Since there is a constant $c=c(y)$ such that $N\sin\left(\frac{y}{N}\right)>c(y)$, for all $N$, and
		$$
		N\tanh\left(\frac{x}{2}+i\frac{y}{2N}\right)=\frac{N\sinh(x)+iN\sin\left(\frac{y}{N}\right)}{\cosh(x)+\cos\left(\frac{y}{2N}\right)}
		,$$
		we get that there is another constant $\tilde{c}=\tilde{c}(y)>0$ for any $|x|\le \pi$ and any $N$ 
		\bae\label{eq:blowpart}
		\left|\sqrt{1-\frac{\lambda^2}{4N^2}-\frac{\lambda^2}{4N^2\tanh^2\left(\frac{x}{2}+i\frac{y}{2N}\right)}}\right|>\frac{1}{\tilde{c}(y)}.
		\eae
		Thus for every $\zeta>0$, there is a $C_\zeta$, such that 
		\begin{equation}\label{eq:tailpart}
		\int_0^{\zeta}\left|\frac{1}{\sqrt{1-\frac{\lambda^2}{4N^2}-\frac{\lambda^2}{4N^2\tanh^2\left(\frac{x+iy}{2N}\right)}}}-1\right|^2dx<C_\zeta.
		\end{equation}
		
		Now for any $x>\zeta$, there is another constant $C'_\zeta(y)$ such that by substituting $x/N$ with $x$ we obtain
		\bae
		&\int_{\zeta/N}^{\pi}N\left|\frac{1}{\sqrt{1-\frac{\lambda^2}{4N^2}-\frac{\lambda^2}{4N^2\tanh^2\left(\frac{x}{2}+\frac{iy}{2N}		\right)}}}-1\right|^2dx\le\\
		&\frac{C'_\zeta(y)}{N^3}\int_{0}^{\pi}\left| \frac{\lambda^2}{4}-\frac{\lambda^2}{4\tanh^2\left(\frac{x}{2}\right)}  \right|^2dx
		,\eae
		which converges to zero as $N\to\infty$ and we conclude the statement 
\end{proof}


\section{} \label{AppA}
Here we will provide the more calculation based proofs.

\begin{proof}[Proof of Lemma \ref{lem:rate}]	
		Let $z=x+iy$ for $x,y\in\mathbb{R}$ s.t. $y\to\infty$. Denote $\zeta:=e^{-iz/N}$ and note that $\zeta\to\infty$ as $y\to\infty$. Then, using Taylor series for approximation we get
					\begin{equation*} 
			\begin{aligned}
				&
				z \overset{f_N} {\to} e^{\frac{-iz}{N}} = \zeta \overset{g} {\to} i\frac{\zeta-1}{\zeta +1} = i \left( 1- \frac{2}{\zeta +1}\right) \\
				& 
				\overset{\slit ^\delta _0}{\to} \left(-\left(1-\frac{4}{\zeta+1}\right)^2 (1-\delta^2)-\delta^2 \right)^{\frac{1}{2}} =  \left(1-\frac{4}{\zeta+1} (1-\delta^2)+o\left(\frac{1}{\zeta}\right)\right)^{\frac{1}{2}} = \\
				&
				= i-\frac{2i}{\zeta +1}(1-\delta^2)+o\left(\frac{1}{\zeta}\right) 
				\overset{g^{-1}}{\to} \frac{2i-\frac{2i}{\zeta +1}(1-\delta^2)+o\left(\frac{1}{\zeta}\right)}{ \frac{2i}{\zeta +1}(1-\delta^2)+o\left(\frac{1}{\zeta}\right)} = \\
				&
				= \frac{1}{ \frac{1-\delta^2}{\zeta +1}+o\left(\frac{1}{\zeta}\right)} - 1 = \frac{\zeta +1}{1-\delta^2+o(1)}-1 \approx \frac{\zeta +1}{1-\delta^2}-1
				\overset{f_N^{-1}} {\to} iN\log\left(\frac{\zeta +1}{1-\delta^2}-1\right) = \\
				& = iN\log \left( \frac{\zeta}{1-\delta^2}\left(1+\frac{1}{\zeta}-\frac{1-\delta^2}{\zeta}\right)\right)= iN\log \left(\frac{\zeta}{1-\delta^2}\right)+iN\log (1+\frac{\delta^2}{\zeta})= \\
				&
				= iN\log(\zeta)-iN\log(1-\delta^2)+iN\frac{\delta^2}{\zeta}+o\left(\frac{1}{\zeta}\right) \\
			\end{aligned}
		\end{equation*}
	Substituting back $\zeta=e^{-iz/N}$ we obtain
	$$
	\slitN_0 (z)=z-iN\log (1-\delta^2)+iN\frac{\delta^2}{e^{-iz\backslash N}}+o\left(e^{-|z|}\right)
	$$
	Now take $\delta=\delta(N,\lambda)\sim \frac{\lambda}{2N}$ for $N$ sufficiently large and get the result:
	\begin{equation*}
			\slitN_0 (z) = z-iN\log\left(1-\frac{\lambda^2}{4N^2}\right)+i\frac{\lambda^2}{4Ne^{-iz\backslash N}}+o\left(e^{-|z|}\right) = z+i\frac{\lambda^2}{4N}+o\left(\frac{1}{N}\right)+o\left(e^{-|z|}\right)
	\end{equation*}
	\end{proof}

	\begin{proof}[Proof of Lemma \ref{lem:conv_slit_fun}]
		Let $z\in\BH$ and consider $ \slitN_0(z)$ for $\delta=\delta(N,\lambda)$. Then, using Taylor approximation:
			\begin{equation*} 
				\begin{aligned}
					z
					&
					\overset{f_N} {\to} e^{\frac{-iz}{N}} \approx 1-\frac{iz}{N} \\
					&
					\overset{g}{\to} \frac{i(\frac{-iz}{N})}{2-\frac{iz}{N}} =
					\frac{z}{2N} \left( \frac{1}{1-iz \backslash 2N} \right)= 
					\frac{z}{2N} \left(1+\frac{iz \backslash2N}{1-iz \backslash 2N}  \right) =
					\frac{z}{2N} + O\left(\frac{1}{N^2}\right) \\
					&
					\overset{\slit^{\delta (N,\lambda)}_0}{\longrightarrow} \left[ \left( \frac{z}{2N}+O\left(\frac{1}{N^2}\right) \right) ^2 \left( 1- \left( 1-\frac{2}{1+e^{\lambda \backslash N}} \right) ^2 \right) - \left( 1-\frac{2}{1+e^{\lambda \backslash N}} \right) ^2\right]^{\frac{1}{2}}  = \\
				  	&
					= \left[ \left( \frac{z}{2N}+O\left(\frac{1}{N^2}\right) \right) ^2 \left( 1-\left( 1-\frac{1}{1+\lambda 	\backslash 2N} \right) ^2 \right) - \left( 1-\frac{1}{1+\lambda \backslash 2N} \right) ^2\right]^{\frac{1}{2}}= \\	
					&
					= \left[ \left( \frac{z}{2N}+O\left(\frac{1}{N^2}\right) \right) ^2 \left( 1-\left(\frac{\lambda}{2N} \right) ^2 \right) -\left( \frac{\lambda}{2N} \right) ^2\right]^{\frac{1}{2}} = \\
					&
					= \left[ \left( \frac{z}{2N}+O\left(\frac{1}{N^2}\right)\right) ^2 - \left( \frac{z\lambda}{4 N^2} + O\left(\frac{1}{N^3}\right) \right) ^2  -  \left( \frac{\lambda}{2N} \right) ^2\right]^{\frac{1}{2}} = \\
					&
					= \left[ \left( \frac{z}{2N}\right) ^2+O\left(\frac{1}{N^3}\right) - \left( \frac{\lambda}{2N} \right) ^2\right]^{\frac{1}{2}} = 
			\frac{1}{2N} \sqrt{z^2- \lambda ^2} + O\left(\frac{1}{N^2}\right)\\
			\end{aligned}
			\end{equation*}
	Set $w := z^2 - \lambda ^2$ and so we'll get
	$$ \slit ^ {\delta(N,\lambda)}_0\circ g\circ f_N (z) \approx \frac{1}{2N} w^{\frac{1}{2}}+O\left(\frac{1}{N^2}\right) .$$
	Coming back to our computations we obtain
	
	\begin{equation*}
		\begin{aligned}
			\frac{1}{2N} w^{\frac{1}{2}} 
			&
			\overset{g^{-1}}{\to} \frac{i+\frac{1}{2N} w^{\frac{1}{2}}+O\left(\frac{1}{N^2}\right)}{i-\frac{1}{2N} w^{\frac{1}{2}}+O\left(\frac{1}{N^2}\right)}
			=1+\frac{ \frac{1}{N}w^{\frac{1}{2}} }{i -\frac{1}{2N} w^{\frac{1}{2}}+O\left(\frac{1}{N^2}\right)} = 1-\frac{i\frac{1}{N}w^{\frac{1}{2}} }{1 + i\frac{1}{2N} w^{\frac{1}{2}}+O\left(\frac{1}{N^2}\right)}=  \\
			&
			= 1- i\frac{1}{N}w^{\frac{1}{2}} \left( 1- i\frac{1}{2N}w^{\frac{1}{2}} +O\left(\frac{1}{N^2}\right) \right)
			=1-i\frac{1}{N}w^{\frac{1}{2}} + O\left(\frac{1}{N^2}\right)+O\left(\frac{1}{N^2}\right) \\
			&
			\overset{f_N^{-1}} {\to} iN \log \left(1-iw^{\frac{1}{2}} \backslash N + O\left(\frac{1}{N^2}\right) \right) \approx iN \left(iw^{\frac{1}{2}} \backslash N + O\left(\frac{1}{N^2}\right)\right) = \\
			&
			= w^{\frac{1}{2}} + O\left(\frac{1}{N}\right) \\
		\end{aligned}
	\end{equation*}
	Then, there is a constant $C(z)$ s.t.
		$$
		\left|\slitN_0(z)-\sqrt{z^2-1}\right| \leq \frac{C(z)}{N}
		$$
	\end{proof}
	
\subsection*{Acknowledgements}

We would like to thank Noam Berger for insightful discussions that were greatly instrumental for the definition of the CHL$^N$ and the proof of the main Theorem.

	\bibliography{career}
	\bibliographystyle{plain}

\end{document}

%% file: conformal.tex
\begin{tikzpicture}
\tikzstyle{redcirc}=[circle,
draw=black,fill=myred,thin,inner sep=0pt,minimum size=3mm]
\tikzstyle{bluecirc}=[circle,
draw=black,fill=myblue,thin,inner sep=0pt,minimum size=5mm]
\usetikzlibrary{patterns}

\node (v1) at (-1,-0.3)  {\tiny{$-N\pi$}};
\node (v2) at (1,-0.3)  {\tiny{$N\pi$}};
\draw [ultra thick] (-1,0) to (1,0);
       (-1,2) -- (-1,0) -- (1,0)-- (1,2) -- cycle;
\draw[pattern=north west lines, pattern color=blue, opacity=0.2] (-1,0) rectangle (1,3);
\node (v4) at (0,0) [redcirc] {\tiny{$0$}};

\node (v3) at (1.5,1.5) {\tiny{$\overset{f}{\to}$}};

\draw[pattern=north west lines, pattern color=blue, opacity=0.2] (1.8,0) rectangle (4.2,3);
\filldraw[color=black!160, fill=white!5, ultra thick](3,1.5) circle (1);

\node (v5) at (4.6,1.5) {\tiny{$\overset{g}{\to}$}};

\draw[pattern=north west lines, pattern color=blue, opacity=0.2] (5,0) rectangle (8,2);
\draw [ultra thick] (5,0) to (8,0);

\node (v8) at (8.3,1.5) {\tiny{$\overset{\slit^\delta}{\to}$}};

\draw[pattern=north west lines, pattern color=blue, opacity=0.2] (8.7,0) rectangle (11.7,2);
\draw [ultra thick] (8.7,0) to (11.7,0);
\draw [ultra thick] (10.2,0) to (10.2,0.7);

\node (v8) at (10.2,-0.5) {\tiny{${\downarrow}{g^{-1}}$}};

\draw[pattern=north west lines, pattern color=blue, opacity=0.2] (8.7,-4) rectangle (11.7,-1);
\filldraw[color=black!160, fill=white!5, ultra thick](10,-2.5) circle (1);
\draw [ultra thick] (11,-2.5) to (11.4,-2.5);

\node (v8) at (7.9,-2.5) {\tiny{$\overset{f^{-1}}{\leftarrow}$}};

 \begin{scope}[shift={(6,-4)}]

\node (v11) at (-1,-0.3)  {\tiny{$-N\pi$}};
\node (v12) at (1,-0.3)  {\tiny{$N\pi$}};
\draw [ultra thick] (-1,0) to (1,0);
\draw[pattern=north west lines, pattern color=blue, opacity=0.2] (-1,0) rectangle (1,3);
\draw [ultra thick] (0,0) to (0,1.3);
\node (v14) at (0,1.3) [redcirc] {\tiny{$i \lambda$}};
\end{scope}

\draw [-stealth](0,-0.5) -- (4.5,-3);
\node (v15) at (2.7,-1.5)  {$\slitN$};

\node (v6) at (4,1.5) [redcirc] {\tiny{$1$}};
\node (v8) at (10.2,0.7) [redcirc] {\tiny{$i\delta$}};
\node (v7) at (6.5,0) [redcirc] {\tiny{$0$}};
\node (v9) at (11.4,-2.5) [redcirc] {};
\node (v10) at (11.4,-2.1)  {\tiny{$\frac{1+\delta}{1-\delta}$}};

\end{tikzpicture}

%% file: SHL_torus_ver_9.bbl
\begin{thebibliography}{10}

\bibitem{berger2020growth}
Noam Berger, Eviatar~B. Procaccia, and Amanda Turner.
\newblock Growth of stationary hastings--levitov.
\newblock {\em The Annals of Applied Probability}, 32(5):3331--3360, 2022.

\bibitem{garnett2005harmonic}
John~B Garnett and Donald~E Marshall.
\newblock {\em Harmonic measure}, volume~2.
\newblock Cambridge University Press, 2005.

\bibitem{MR915132}
H.~Kesten.
\newblock Hitting probabilities of random walks on {${\bf Z}^d$}.
\newblock {\em Stochastic Process. Appl.}, 25(2):165--184, 1987.

\bibitem{Kesten1987HowLA}
Harry Kesten.
\newblock How long are the arms in dla.
\newblock {\em Journal of Physics A}, 20, 1987.

\bibitem{meakin1983diffusion}
Paul Meakin.
\newblock Diffusion-controlled deposition on fibers and surfaces.
\newblock {\em Physical Review A}, 27(5):2616, 1983.

\bibitem{mu2019scaling}
Yingxin Mu, Eviatar~B. Procaccia, and Yuan Zhang.
\newblock Scaling limit of dla on a long line segment.
\newblock {\em Transactions of the American Mathematical Society}, 2022.

\bibitem{norris2012hastings}
J.~Norris and A.~Turner.
\newblock Hastings--levitov aggregation in the small-particle limit.
\newblock {\em Communications in Mathematical Physics}, pages 1--33, 2012.

\bibitem{procaccia2021dimension}
Eviatar~B. Procaccia and Itamar Procaccia.
\newblock Dimension of diffusion-limited aggregates grown on a line.
\newblock {\em Physical Review E}, 103(2):L020101, 2021.

\bibitem{procaccia2018stationary}
Eviatar~B. Procaccia, Jiayan Ye, and Yuan Zhang.
\newblock Stationary harmonic measure as the scaling limit of truncated
  harmonic measure.
\newblock {\em arXiv preprint arXiv:1811.04793}, 2018.

\bibitem{procaccia2020stationary}
Eviatar~B. Procaccia, Jiayan Ye, and Yuan Zhang.
\newblock Stationary dla is well defined.
\newblock {\em Journal of Statistical Physics}, 181(4):1089--1111, 2020.

\bibitem{procaccia2019stationary}
Eviatar~B. Procaccia and Yuan Zhang.
\newblock Stationary harmonic measure and dla in the upper half plane.
\newblock {\em Journal of Statistical Physics}, 176(4):946--980, 2019.

\bibitem{procaccia2021sets}
Eviatar~B. Procaccia and Yuan Zhang.
\newblock On sets of zero stationary harmonic measure.
\newblock {\em Stochastic Processes and their Applications}, 131:236--252,
  2021.

\end{thebibliography}
